\documentclass[11pt]{amsart}

\usepackage{enumerate,url,amssymb,  mathrsfs}

\newtheorem{theorem}{Theorem}[section]
\newtheorem{lemma}[theorem]{Lemma}
\newtheorem*{lemma*}{Lemma}
\newtheorem{proposition}[theorem]{Proposition}

\theoremstyle{definition}

\theoremstyle{remark}
\newtheorem{remark}[theorem]{Remark}

\numberwithin{equation}{section}

\newcommand{\abs}[1]{\lvert#1\rvert}

\def\XXint#1#2#3{{\setbox0=\hbox{$#1{#2#3}{\int}$}
\vcenter{\hbox{$#2#3$}}\kern-.5\wd0}}

\def\le{\leqslant}
\def\ge{\geqslant}

\begin{document}


\title[Subharmonicity and quasiregular harmonic mappings]{Subharmonicity of the modulus of quasiregular harmonic mappings}
\subjclass{Primary 31A05; Secondary 31B10}

\keywords{Quasiregular harmonic mappings, Subharmonic functions}
\author{David Kalaj}
\address{University of Montenegro, Faculty of Natural Sciences and
Mathematics, Cetinjski put b.b. 8100 Podgorica, Montenegro}
\email{davidk@t-com.me} 
\author{ Vesna Manojlovi\' c}
\address{
University of Belgrade, Faculty of Organizational Sciences, Jove
Ili\'ca 154, Belgrade, Serbia} \email{vesnam@fon.rs}

\maketitle

\begin{abstract}
In this note we determine all numbers $q\in \mathbf R$ such that
$|u|^q$ is a subharmonic function, provided that $u$ is a
$K-$quasiregular harmonic mappings in an open subset $\Omega$ of the
Euclidean space $\mathbf R^n$.
\end{abstract}
\section{Introduction} By
$|\cdot |$ we denote the Euclidean norm in $\mathbf R^n$ and let
$\Omega$ be a region in  $\mathbf R^n$. In this paper we consider
$K$-quasiregular harmonic mappings, where $ K \ge 1$. We recall that
a harmonic mapping $u(x)=(u_1(x),\dots, u_n(x)):\Omega\to \mathbf
R^n$ with formal differential matrix
$$Du(x)=\{\partial_iu_j(x)\}_{i,j=1}^n$$ is $K$-quasiregular if
\begin{equation}\label{bege}K^{-1}|D u(x)|^n\le J_u(x)\le K l(D u(x))^n,\quad\text{for all $x\in \Omega$,}\end{equation} where $J_u$
is the Jacobian of $u$ at $x$, $$|D u|:=\max\{|Du(x) h|: |h|=1\},$$
and
$$l(D u):=\min\{|D u(x) h|: |h|=1\}.$$ See \cite[p.~128]{vo} for the definition of quasiregular mappings in more general setting. A quasiregular homeomorphism is called
quasiconformal.

Let $0<\lambda_1^2\le \lambda_2^2\le \dots\le \lambda_n^2$ be the
eigenvalues of the matrix $D u(x)D u(x)^t$. Here  $D u(x)^t$ is the
transpose of the matrix $Du(x)$.  Then
\begin{equation}\label{jaco}J_u(x) = \prod_{k=1}^n
\lambda_k,\end{equation} \begin{equation}|D
u|=\lambda_n\end{equation} and
\begin{equation}
l(D u)=\lambda_1.
\end{equation}
 For the Hilbert-Schmidt norm of the
matrix $D u(x)$, defined by $$\|Du(x)\|=\sqrt{\mathrm{Trace}(D u(x)D
u(x)^t)}$$ we have

\begin{equation}\label{clare}
\|Du(x)\|=\sqrt{\sum_{k=1}^n\frac{\partial u}{\partial
x_k}\bullet\frac{\partial u}{\partial
x_k}}=\sqrt{\sum_{k=1}^n\left|\frac{\partial u}{\partial
x_k}\right|^2}
\end{equation}
and
\begin{equation}\|Du(x)|=\sqrt{\sum_{k=1}^n \lambda_k^2}.
\end{equation}
Here $\bullet$ denotes the inner product between vectors. From
\eqref{bege}, for a quasiregular mapping we have
\begin{equation}\label{epa}\frac{\lambda_n}{\lambda_k},\frac{\lambda_k}{\lambda_1}\le
K,\ \ k=1,\dots,n.\end{equation} It is well known that if
$u=(u_1,\dots,u_n)$ is a harmonic mapping defined in a region
$\Omega$ of the Euclidean space $\mathbf R^n$, then $|u |^p$ is
subharmonic for $p\ge 1$, and that, in the general case, is not
subharmonic for $p < 1$. Let us prove this well-known fact. If $u$
is harmonic, then by a result in \cite[Lemma 1.4.]{jmaa} (see also
\cite[Eq.~(4.9)-(4.11)]{k})
$$\Delta|u|=|u|\left\|D\left(\frac{u}{|u|}\right)\right\|^2.$$ So $\Delta|u|\ge 0$ for those points $x$, such that $u(x)\neq
0$. If $u(a)=0$, then we consider the harmonic mapping
$u_m(x)=u(x)+(1/m,0,\dots,0)$. Then $u_m(a)\neq 0$, and $\Delta
|u_m(x)|\ge 0$ in some neighborhood of $a$. It follows from the
definition of subharmonic functions that the uniform limit of a
convergent sequence of subharmonic functions is still subharmonic.
Since $|u_m(x)|\to |u(x)|$, it follows that $|u|$ is subharmonic in
$a$. Since the function $g(s)=s^p$, is convex for $p\ge 1$, we
obtain that $|u|^p$ is subharmonic providing that $u$ is harmonic.
(For the above facts we refer to \cite[Chapter~2]{hk}).

Recently, several authors have proved the following two
propositions, which is the motivation for our study.
\begin{proposition}\label{pav}\cite{kpa} If f is a $K$-quasiregular harmonic map in a plane
domain, then $| f |^q$ is subharmonic for $q \ge 1-K^{-2}$.
\end{proposition}
\begin{proposition}\label{amb}\cite{akb} If f is a $K$-quasiregular harmonic map in a
space domain, then $| f |^q$ is subharmonic for some
$q=q(K,n)\in(0,1)$.
\end{proposition}
This paper is continuation of \cite{akb} in which Proposition
\ref{pav} was extended to the $n$-dimensional setting. In \cite{akb}
the authors prove only the existence of an exponent $q\in (0,1)$
without giving the minimal value of $q.$  Here we improve
Proposition \ref{amb} by giving the optimal value of $q.$ Our proof
is completely different from those given in \cite{akb} and
\cite{kpa}. Moreover for the first time we consider the case $q<0$.

Our proof is based on the following well-known explicit computation.
\begin{proposition}\label{ste} \cite[Ch.~VII~3,~p.217]{stein}. Let $u = (u_1, \dots, u_n)
:\Omega\to
 \mathbf R^n,$ be
harmonic, let $\Omega_0 =\Omega\setminus u^{-1}(0)$, let $q \in
\mathbf R$. Then for $x\in\Omega_0$ $$\Delta |u|^q = q\left[
|u|^{q-2}\sum_{k=1}^n\left|\nabla u_k
\right|^2+(q-2)|u|^{q-4}\sum_{k=1}^n\left(u\bullet \frac{\partial
u}{\partial x_k}\right)\right].$$
\end{proposition}
\begin{proof}
 Write $v := |u|^q = (u_1^2 + \dots + u_n^2)^p,$ for $p := q /2$ . A direct
computation gives \[\begin{split}v_{x_1} &= p(u_1^2 + \dots +
u_n^2)^{p-1} \cdot (2u_1{u_1}_{x_1} + \dots + 2u_n {u_n}_{x1}) \\&=
q(u_1^2 + \dots + u_n^2)^{p-1} \cdot (u_1{u_1}_{x_1} + \dots + u_n
{u_n}_{x1}),\end{split}\] and further \[\begin{split}v_{x_1x_1} &=
q\{2(p - 1)(u_1^2 + \dots + u_n^2)^{p-2} \cdot (u_1{u_1}_{x_1} +
\dots + {u_n}_{x_1})^2+\\ &+ (u_1^2 + \dots + u_n^2)^{p-1}\cdot
[u_1{u_1}_{x_1x_1} + ({u_1}_{x_1})^2 + \dots + u_n {u_n}_{x_1x_1} +
({u_n}_{x_1})^2]\}.\end{split}\] Therefore
\[\begin{split}\Delta v &= v_{x_1x_1} + \dots + v_{x_nx_n}
\\&= q\{|u|^{q-2}[(u_1\Delta u_1 + \dots + u_n\Delta u_n) +
(\sum_{k=1}^n {u_1}_{x_k}^2 + \dots\\& \ \ \ \ \  + \sum_{k=1}^n
{u_n}_{x_k}^2 )] + (q - 2)|u|^{q-4}\sum_{k=1}^n \left(\sum_{j=1}^n
u_j {u_j}_{x_k}\right)^2\}
\\&= q\{|u|^{q-2}( \sum_{k=1}^n {u_1}_{x_k}^2 + \dots +\sum_{k=1}^n
{u_n}_{x_k}^2) + (q - 2)|u|^{q-4} \sum_{k=1}^n\left(\sum_{j=1}^n u_j
\cdot \frac{\partial u_j}{\partial x_k}\right)^2\} \\&=
q|u|^{q-4}\{|u|^2 \sum_{j=1}^n\left( \sum_{k=1}^n {u_j}_{x_k}^2
\right) + (q - 2) \sum_{k=1}^n \left(\sum_{j=1}^n u_j \cdot
\frac{\partial u_j}{\partial x_k}\right)^2\} \\&= q|u|^{q-4}\{|u|^2
\sum_{j=1}^n |\nabla u_j |^2 + (q - 2) \sum_{k=1}^n \left(u \bullet
\frac{\partial u}{\partial x_k}\right)^2\}.\end{split}\]
\end{proof}
\section{Main result}
\begin{theorem}\label{mai}
Let $u$ be $K-$quasiregular harmonic in $\Omega\subset \Bbb R^n$.
Then the mapping $g(x) =|u(x)|^q$ is subharmonic in

\begin{enumerate}\item{} $\Omega$ for $q\ge\max\{ 1-\frac{n-1}{K^2},0\}$;
\item{}
$\Omega\setminus u^{-1}(0)$, for $q\le 1-(n-1){K^2}$.
\end{enumerate}
Moreover for $1-(n-1){K^2}<q<1-\frac{n-1}{K^2}$, there exists a
$K$-quasiconformal harmonic mapping such that $|u|^q$ is not
subharmonic.
\end{theorem}
\begin{remark}
If $n=2$ then $1-\frac{n-1}{K^2}=1-K^{-2}.$ Thus Theorem~\ref{mai}
is an extension of Proposition~\ref{pav}.
\end{remark}
\begin{remark}
In the case $1\le K\le \sqrt{n-1}$ the function $|u|^q$ is
subharmonic for all $q>0.$
\end{remark}
\begin{proof}[Proof of Theorem~\ref{mai}] Let us fix such a map $u:\Omega\to \mathbf R^n$ and set $\Omega_0 = \Omega\setminus
u^{-1}\{0\}.$ We have to find all positive real numbers $q$ such
that $\Delta |u|^q\ge 0$ on $\Omega_0$. Since $u$ is quasiregular,
the set $Z = \{x\in \Omega_0 : \det Du(x) = 0\}$ has measure zero
(see \cite{vo}), it is also closed since $u$ is smooth. In
particular, $\Omega_1 = \Omega_0\setminus Z$ is dense in $\Omega_0$
and thus it suffices to prove that $\Delta |u|^q \ge 0$ on
$\Omega_1$. From Proposition~\ref{ste}, we obtain
\begin{equation}
\label{laplacian} \Delta |u|^q = q\left[ |u|^{q-2}\|D
u\|^2+(q-2)|u|^{q-4}\left|\sum_{j=1}^n u_j \nabla
u_j\right|^2\right].
\end{equation}
So we find all real $q$ such that

$$q\left(|u|^{q-2}\|D
u\|^2+(q-2)|u|^{q-4}\left|\sum_{j=1}^n u_j \nabla
u_j\right|^2\right)\ge 0.$$ If $q\ge 2$, then $\Delta |u|^q\ge 0$.
Assume that $q\ge 0$ and $q<2$ such that
$$ \left|\sum_{j=1}^n u_j(x) \nabla u_j(x)\right|^2\le \frac{1}{2-q}|u(x)|^2\|D
u(x)\|^2,\ \ \ x\in \Omega_1.$$

After normalization, we see that it suffices to find all constants
$q <2$ such that

\begin{equation}\label{op}\sup_{z\in S^{n-1}} \left|\sum_{j=1}^n z_j \nabla u_j(x)\right|^2\le
\frac{1}{2-q}\|D u(x)\|^2,\ \ \ x\in \Omega_1.\end{equation}

Let $0<\lambda_1^2\le \lambda_2^2\le \dots\le \lambda_n^2$ be the
eigenvalues of the matrix $D u(x)D u(x)^t$. Then
\begin{equation}\label{po}
\sup_{z\in S^{n-1}} \left|\sum_{j=1}^n z_j \nabla
u_j(x)\right|^2=\lambda_n^2
\end{equation}
\begin{equation}\label{jo}
\inf_{z\in S^{n-1}} \left|\sum_{j=1}^n z_j \nabla
u_j(x)\right|^2=\lambda_1^2\end{equation} and
\begin{equation}\label{yes}
\|D u(x)\|^2= \sum_{k=1}^n \lambda_k^2.
\end{equation}
Because $u$ is $K-$quasiregular from \eqref{epa} we have
\begin{equation}\label{qe}\frac{\lambda_n}{\lambda_k}\le K,\ \ k=1,\dots, n-1.\end{equation} Thus \eqref{op} can be written as
\begin{equation}\label{opo}\lambda_n^2\le \frac{1}{2-q}\sum_{k=1}^n \lambda_k^2.\end{equation}
By \eqref{yes} and \eqref{qe} we get that, the inequality
\eqref{opo} is satisfied whenever
\begin{equation}
\frac{1}{1+\frac{n-1}{K^2}}\le \frac{1}{2-q}
\end{equation}
i.e.
\begin{equation}
\max\left\{0,1-\frac{n-1}{K^2}\right\}\le q<2.
\end{equation}
If $q<0$, then we should have \begin{equation}\label{nop}\inf_{z\in
S^{n-1}} \left|\sum_{j=1}^n z_j \nabla u_j(x)\right|^2\ge
\frac{1}{2-q}\|D u(x)\|^2,\ \ \ x\in \Omega_1,\end{equation} i.e.
\begin{equation*}{2-q}\ge \sum_{k=1}^n
\frac{\lambda_k^2}{\lambda_1^2}.\end{equation*} Because $u$ is
$K-$quasiregular from \eqref{epa}
\begin{equation}\label{lew}\frac{\lambda_k}{\lambda_1}\le K,\ \ k=2,\dots, n.\end{equation}
Thus if
\begin{equation}\label{roma} q\le 1-(n-1)K^2,\end{equation} then
\eqref{nop} holds. To finish the proof we need the following lemma.
\begin{lemma}\label{le}
For any $1-(n-1){K^2}<q<1-\frac{n-1}{K^2}$ there is a (linear)
harmonic $K$-quasiconformal mapping $u$ such that $|u|^q$ is not
subharmonic.
\end{lemma}
\begin{proof}[Proof of Lemma~\ref{le}]
Assume first that $q>0$. We will consider linear mapping $u:\mathbf
R^n\longrightarrow\mathbf R^n$ defined by
\begin{equation}\label{defa}u(x_1,\dots,x_{n-1},x_n)=(x_1,\dots,x_{n-1}, Kx_n),\end{equation} where
$K\ge 1$. It is obviously harmonic and $K-$quasiconformal. If we put
this mapping in formula (\ref{laplacian}) we get
$$
[(n-1)+K^2]|u|^2+(q-2)\left|\sum_{j=1}^{n-1}x_je_j+K^2e_nx_n\right|^2\geq
0
$$
which is equivalent to
$$
(n-1+K^2)\left[\sum_{n=1}^{j-1}x_j^2+K^2x_n^2\right]+(q-2)\left|\sum_{j=1}^{n-1}x_je_j+K^2e_nx_n\right|^2\geq
0.
$$
By choosing $x_1=\cdots=x_{n-1}=0$ and $x_n=1$, we obtain
$$
(n-1+K^2)K^2\geq (2-q)K^4
$$
which is equivalent to
$$
q\geq 1-\frac{n-1}{K^2}.
$$
For $q<0$ we consider the linear mapping $u:\mathbf
R^n\longrightarrow\mathbf R^n$ defined by
\begin{equation}\label{dela}u(x_1,\dots,x_{n-1},x_n)=(x_1,\dots,x_{n-1}, x_n/K).\end{equation}
\end{proof}
To finish the proof we only need to take $\tilde u=u|_{\Omega}$,
where $u$ is defined in \eqref{defa} respectively in \eqref{dela}.
\end{proof}
\subsection*{Acknowledgment}
We thank the referee for providing constructive comments and help in
improving the contents of this paper.

\end{document}